\documentclass[11pt]{amsart}
\usepackage[margin=1.4 in]{geometry}

\usepackage{amsmath, amscd}
\usepackage{mathrsfs}
\usepackage{amsfonts}
\usepackage{latexsym,epsfig}
\usepackage{mathabx}
\usepackage{amsmath, amscd, amsthm}
\usepackage[pdftex]{color}
\usepackage{comment}
\usepackage{marginnote}
\usepackage[colorlinks,citecolor=blue,linkcolor=red]{hyperref}
\usepackage[nameinlink]{cleveref}

\newcommand{\MCG}[0]{\mathcal{MCG}}

\newcommand{\Out}[0]{\mathrm{Out}}

\newcommand{\PP}[0]{\mathcal{P}}

\newcommand{\R}{\mathbb{R}}

\newcommand{\im}[0]{\text{im }}

\newcommand{\m}{\mathfrak{m}}

\renewcommand{\sl}{\mathrm{sl}}
\newcommand{\cone}{\mathrm{cone}}
\newcommand{\cusp}{\mathrm{cusp}}

\usepackage{hyperref}

\newtheorem{theorem}{Theorem}[section]
\newtheorem{lemma}[theorem]{Lemma}
\newtheorem{proposition}[theorem]{Proposition}
\newtheorem{corollary}[theorem]{Corollary}
\newtheorem*{claim*}{Claim}
\newtheorem{claim}{Claim}
\newtheorem{definition}{Definition}

\theoremstyle{definition}
\newtheorem{remark}[theorem]{Remark}

\newtheorem{question}{Question}

\newtheorem{example}{Example}

\newcommand{\F}{\mathcal{F}}

\newcommand{\T}{\mathcal{T}}
\newcommand{\X}{\mathcal{X}}

\newcommand{\orb}{\mathrm{orb}_x}
\newcommand{\define}[1]{\textbf{#1}}

\renewcommand{\P}{\mathcal{P}}
\newcommand{\diam}{\mathrm{diam}}
\renewcommand{\im}{\mathrm{im}}

\newcommand{\symd}{d_{\mathrm{sym}}}

\begin{document}

\title[Pulling back stability]{Pulling back stability with applications to $\mbox{Out}(F_{n})$ and relatively hyperbolic groups}

\author[T. Aougab]{Tarik Aougab}
\address{Department of Mathematics\\
Brown University\\
151 Thayer\\
Providence, RI 02912, USA\\}
\email{\href{mailto:tarik\_aougab@brown.edu}{tarik\_aougab@brown.edu}}
\author[M.G. Durham]{Matthew Gentry Durham}
\address{Department of Mathematics\\
University of Michigan\\
530 Church St\\
Ann Arbor, MI 48109, USA\\}
\email{\href{mailto:durhamma@umich.edu}{durhamma@umich.edu}}
\author[S.J. Taylor]{Samuel J. Taylor}
\address{Department of Mathematics\\ 
Temple University\\ 
1805 North Broad St\\ 
Philadelphia, PA 19122, USA\\}
\email{\href{mailto:samuel.taylor@temple.edu}{samuel.taylor@temple.edu}}
\date{\today}

\maketitle

\begin{abstract} 
We prove that stability -- a strong quasiconvexity property -- pulls back under proper actions on proper metric spaces. This result has several applications, including
that convex cocompact subgroups of both mapping class groups and  
outer automorphism groups of free groups are stable. We also characterize stability 
in relatively hyperbolic groups whose parabolic subgroups have linear divergence.
\end{abstract}

\section{Introduction}

The concept of hyperbolicity has been central to the study of finitely generated groups, since hyperbolic groups satisfy a host of useful algebraic, geometric, and algorithmic properties \cite{Gromov:essay, alonso1991notes, ghys1990groupes}. Thus if a geodesic metric space $X$ is not globally hyperbolic, it becomes natural to look for the subspaces or directions along which $X$ does exhibit negatively curved behavior. 

The aim of this paper is to study a negatively curved behavior called \textit{stability}, which is a generalization of the notion of quasiconvexity in hyperbolic groups.
Informally, 
an undistorted, quasiconvex subspace $Y$ of a geodesic metric space $X$ is called \define{stable} if any two quasigeodesics in $X$ with common endpoints in $Y$ are forced to uniformly fellow travel. A subgroup $H$ of a finitely generated group $G$ is stable 
when it constitutes a stable subset of a Cayley graph for $G$.
We note that stable subgroups are always hyperbolic and quasiconvex, and subgroup stability is a quasi-isometry invariant.

We prove that stability pulls back under proper actions on proper spaces (\Cref{th:pull_back}):

\begin{theorem}[Pulling back stability] \label{intro:th_pulling_back} 
Let $G$ be a finitely generated group with a proper action $G \curvearrowright X$ on a proper geodesic metric space $X$. Let $H \le G$ be 
such that for some $x \in X$, the orbit map $\orb \colon G \to X$ given by $g \mapsto gx$ restricts to a stable embedding on $H$. Then $H$ is stable in $G$.
\end{theorem}

Our main applications are to establishing criteria for subgroup stability in subgroups of mapping class groups,
 outer automorphism groups of free groups (\Cref{intro:th_mod_out}), and relatively hyperbolic groups (\Cref{intro:th_stab_char}).

The stability property is a generalization of the more familiar \textit{Morse stability property} for a single quasigeodesic. For a function $D \colon \mathbb{R}_{\ge  0}^{2} \rightarrow \mathbb{R} _{\ge 0}$, a quasigeodesic $\gamma$ in a geodesic metric space $X$ is called \define{$D$-stable} if  any $(\kappa, \lambda)$-quasigeodesic $g$ with endpoints on $\gamma$ remains within the $D(\kappa, \lambda)$-neighborhood of $\gamma$. The connection between stability for an undistorted subgroup $H$ and the stability property for a single quasigeodesic is as follows: $H$ is stable in $G$ if and only if there exists $D$ so that every geodesic in $H$ is $D$-stable in $G$. 
That is, stability of a subgroup equates to \textit{uniform} stability for its geodesics in the ambient group. 

A main ingredient in the proof of \Cref{intro:th_pulling_back} is to demonstrate an alternative characterization of Morse stability, called \textit{middle recurrence}, which was first introduced by Drutu--Mozes--Sapir \cite{DMS}. This characterization satisfies two important properties: 
\begin{enumerate}
\item It behaves well under Lipschitz maps (e.g., orbit maps of finitely generated groups into metric spaces).
\item It constitutes an \textit{effective} characterization of stability; that is, we prove an explicit relationship between the stability function $D(\kappa, \lambda)$ and the function associated with the middle recurrence characterization. As stability of a subgroup requires a simultaneous control of the stability functions for all geodesics of that subgroup, this effectiveness plays a crucial role. 
\end{enumerate}

A quasigeodesic $q \subset X$ is called \define{middle recurrent} if for any $C> 0$, there exists $K\geq 0$ so that if $p$ is any path with endpoints $a,b$ on $q$ whose arclength is less than $ C \cdot d_{X}(a,b)$, then $p$ meets the $K$-neighborhood of the ``middle portion'' of $q$ between $a$ and $b$.  A precise definition of ``middle portion'' can be found in \Cref{sec:mr and stab}, but heuristically this is a subset of $q$ which lies at a definite distance from both endpoints $a$ and $b$. Note that the condition $\ell(p) \leq C \cdot d_{X}(a,b)$ does not impose any local constraint on the path $p$; in particular, it need not be a quasigeodesic. We prove the following:

\begin{theorem}[Middle recurrence and stability] \label{th:intro_middle_recurrence}
Let $q$ be a quasigeodesic in a geodesic metric space $X$. Then $q$ is stable if and only if $q$ is middle recurrent. Moreover, its recurrence function can be bounded from above only in terms of its stability function, and visa versa.
\end{theorem}

See Subsection \ref{sec:Mid_Lit} for a discussion of middle recurrence in \cite{DMS}.  In forthcoming work, the authors establish effective control on the \emph{divergence} of a stable quasigeodesic in terms of its stability function. That these notions are related follows from \cite{DMS, CharSult, ACGH}.

\subsection*{Applications.}

Our motivation for \Cref{intro:th_pulling_back} lies in its applications---the study of stable subgroups of the mapping class group, the outer automorphism group of the free group, and relatively hyperbolic groups.

\subsubsection*{$\MCG(S)$ and $\Out(F_n)$}


 Given an orientable surface $S$ of finite type, let $\MCG(S)$ denote its mapping class group and let $\mathcal{T}(S)$ denote its Teichm{\"u}ller space. Motivated by the many analogies between the action of $\MCG(S)$ on $\mathcal{T}(S)$ and the action of a Kleinian group on hyperbolic space, Farb--Mosher \cite{FarbMosher} defined a subgroup $H \le \MCG(S)$ to be \emph{convex cocompact} if for any $X \in \mathcal{T}(S)$, the orbit $ H \cdot X \subset \mathcal{T}(S)$ is quasiconvex with respect to the Teichm\"uller metric.  When $S$ is closed, combined work of Farb-Mosher and Hamenst\"adt \cite{HamCC} proves that such $H$ are precisely those subgroups which determine hyperbolic surface group extensions.

Kent--Leininger \cite{KentLein} and independently Hamenst\"adt \cite{HamCC} later proved that convex cocompactness is equivalent to the condition that any orbit map of $H$ into the \textit{curve complex} $\mathcal{C}(S)$ is a quasi-isometric embedding.
 The second and third authors \cite{DT15} subsequently gave a characterization of convex cocompactness that is \textit{intrinsic} to $\MCG(S)$, in that it does not reference an action on some external space: \emph{$H$ is convex cocompact if and only if $H$ is stable in $\MCG(S)$}. 

Let $\Out(F_{n})$ denote the outer automorphism group of the free group $F_{n}$ on $n\geq 3$ generators. Motivated by the above results, Hamenst\"adt--Hensel \cite{hamenstadt2014convex} proposed the definition that $H \le \Out(F_{n})$ is \emph{convex cocompact} if any orbit map into the \textit{free factor complex} $\mathcal{F}_n$ (see \Cref{subsec:out} for definitions) is a quasi-isometric embedding.  
Dowdall--Taylor \cite{dowdall2014hyperbolic} proved that such subgroups, if purely atoroidal, determine hyperbolic extensions of $F_n$, analogous to the situation in $\MCG(S)$.

Using \Cref{intro:th_pulling_back} and work of Dowdall--Taylor \cite{dowdall2014hyperbolic,dowdall2015contracting}, we recover one direction of the main theorem of \cite{DT15} and extend it to $\Out(F_{n})$, using one unified approach (\Cref{thm:stab_map} and \Cref{th: Out}):

\begin{theorem} \label{intro:th_mod_out} If $H\le \MCG(S)$ quasi-isometrically embeds into $\mathcal{C}(S)$, then $H$ is stable in $\MCG(S)$. Similarly, if $G\le \Out(F_{n})$ quasi-isometrically embeds into $\mathcal{F}_n$, then $G$ is stable in $\Out(F_{n})$. Thus in both cases, convex cocompactness implies stability. 
\end{theorem}

We note that Hamenst\"adt has announced a complete characterization of stable subgroups of $\Out(F_{n})$; according to this announcement, the converse of \Cref{intro:th_mod_out} in the setting of $\Out(F_{n})$ does not hold \cite{HamAnounce}. 

\subsubsection*{Relatively hyperbolic groups} Mirroring the situation in $\MCG(S)$ and $\Out(F_{n})$, we also obtain a criterion for stability in a relatively hyperbolic group; in this result the cusped space \cite{GM08} plays the role of the curve and free factor complexes (see \Cref{subsec:relhyp} for the relevant definitions). 

\begin{theorem}[\Cref{prop:pb_relhyp}] \label{intro:th_stab_rel_crit} Suppose that $G$ is hyperbolic relative to a family of subgroups $\mathcal{P}$,  
and that $H$ is a subgroup of $G$.
If an orbit map of $H$ into the cusped space $\cusp(G, \mathcal{P})$ is a quasi-isometric embedding, then $H$ is stable in $G$.
\end{theorem}

\Cref{intro:th_stab_rel_crit} can be promoted to a full characterization of stable subgroups when the extra assumption of linear divergence is placed on the peripheral subgroups $\mathcal{P}$.

\begin{theorem}[\Cref{th:relhyp_char_stab}] \label{intro:th_stab_char} Let $(G, \mathcal{P})$ be relatively hyperbolic and suppose that each $P \in \P$ is 1-ended with linear divergence. Then 
for any $H \le G$,
the following are equivalent:
\begin{enumerate}
\item $H$ is stable in $G$.
\item $H$ has a quasi-isometric orbit map into $\cusp(G, \mathcal{P})$.
\item $H$ has a quasi-isometric orbit map into the coned-off Cayley graph $\cone(G, \mathcal{P})$. 
\end{enumerate}
\end{theorem}

\subsubsection*{Inheriting subgroup stability}

It is a fairly immediate consequence of the definitions that if $H \le L \le G$ with all three finitely generated, and $H$ is stable in $L$ and $L$ is stable in $G$, then $H$ must be stable in $G$. Thus, stability is transitive under subgroup inclusion. Using \Cref{intro:th_pulling_back}, we prove that stability is \textit{inherited} under subgroup inclusion as well:

\begin{theorem} \label{intro:th_inherit} Let $H \le L \le G$, with $H,L$ finitely generated, and suppose that $H$ is stable in $G$. Then $H$ is stable in $L$. 
\end{theorem}

\begin{proof}
Choose a finite generating set $S$ for $G$ that includes finite generating sets for $L$ and $H$.  Then the orbit of $H$ in the corresponding Cayley graph $\Gamma(G, S)$ is stable.  Since $\Gamma(G, S)$ and the $L$-action on it are both proper, \Cref{intro:th_pulling_back} implies that $H$ is stable in $L$.
\end{proof}

 In the special case where $H$ is cyclic, \Cref{intro:th_inherit} appears as Lemma  $3.25$ in \cite{DMS}. However, this lemma relied on an error earlier in that paper; see \Cref{sec:Mid_Lit} for further discussion.
\begin{remark} \label{intro:rem_inherit_distort} 
\Cref{intro:th_inherit} is most interesting when $L$ is highly distorted in $G$. 
For example, we may take $G$ to be the mapping class group and $L$ to be the handlebody or Torelli subgroup. Similarly, if $G = \Out(F_n)$, we could take $L$ to be the Torelli subgroup.
\end{remark}

\subsubsection*{Random subgroups are stable}

Let $G$ be a finitely generated group, $k \geq 2$, and $\mu$ a probability measure on $G$ whose support generates a non-elementary semigroup.  Let $\Gamma(n) = \langle w^{1}_{n}, w^{2}_{n},...,, w^{k}_{n} \rangle$ be the subgroup generated by the $n^{th}$ step of $k$ independent random walks.

Following Taylor--Tiozzo \cite{TT16}, we say a $k$-generated random subgroup of $G$ has property $P$ if

\[ \mathbb{P}[ \Gamma(n)\mbox{ has $P$} ] \rightarrow 1 \hspace{2 mm} \mbox{as} \hspace{2 mm} n \rightarrow \infty. \] 

Combining \Cref{intro:th_mod_out} with a result of Taylor-Tiozzo \cite{TT16}, we prove random subgroups of several aforementioned groups are stable: 

\begin{corollary} \label{intro:cor_random} 
Let $G$ be $\MCG(S), \Out(F_n),$ nontrivially relatively hyperbolic, a handlebody group, or the Torelli subgroup of $\MCG(S)$ or $\Out(F_n)$.
Then a $k$-generated random subgroup of $G$ is stable.
\end{corollary}

\begin{proof}
In each case, we are given a nonelementary action $G \curvearrowright X$ on a hyperbolic space $X$ such that if $H \le G$ has a quasi-isometric orbit map into $X$, then $H$ is stable in $G$. For $G = \MCG(S)$, this follows from \cite{DT15}, for $G=\Out(F_n)$ this is \Cref{th: Out}, and for $G$ relatively hyperbolic, this is \Cref{prop:pb_relhyp}.

Now we apply the main theorem of \cite{TT16},  to conclude that a random subgroup of $G$ quasi-isometrically embeds into $X$. 
Hence, a random subgroup of $G$ is stable.
\end{proof}

For $\MCG(S)$, \Cref{intro:cor_random} is proven in \cite{TT16}, but it is novel for the other examples.

\subsection*{Stable coherence}
It follows from \Cref{intro:th_inherit} that if $H \le G$ is stable, then so is each cyclic subgroup $\langle h \rangle \le G$ for $h\in H$. 
We say that a group $G$ has the \emph{stable coherence property} if for every finitely generated, undistorted subgroup $H \le G$, $H$ is stable in $G$ whenever each of its nontrivial cyclic subgroups is stable in $G$.

\begin{question} \label{conj:local_to_global}
What groups $G$ have the stable coherence property for stability?
\end{question}

In \cite{koberda2014geometry}, Koberda--Mangahas--Taylor prove that right-angled Artin groups have the stable coherence property. In that case, the assumption that $H$ is undistorted was unnecessary to conclude stability. Recently, Bestvina--Bromberg--Kent--Leininger \cite{BBKL} proved that undistorted purely pseudo-Anosov subgroups of $\MCG(S)$ are convex-cocompact. Translating this result into our language shows that  $\MCG(S)$ also has the stable coherence property. Whether or not the undistorted assumption is necessary for subgroups of the mapping class group remains unknown. 

We thank David Hume for pointing out to us that not all finitely generated groups $G$ have the stable coherence property; for example, see \cite[Theorem 6.6]{ACGH2}. However, there are no known counterexamples when $G$ is finitely presented.

\subsection{Structure of paper}

In \Cref{sec:background}, we begin with some background on metric spaces and overview definitions and equivalent characterizations of stability.  In \Cref{sec:mr and stab}, we prove that stability and middle-recurrence are effectively equivalent.  In \Cref{sec:pbs}, we prove the main result, \Cref{intro:th_pulling_back}, that stability pulls back under proper actions for finitely generated groups on proper geodesic metric spaces.  \Cref{sec:app} holds our applications, with those to $\MCG(S)$ in \Cref{subsec:mod}, $\Out(F_n)$ in \Cref{subsec:out}, and relatively hyperbolic groups in \Cref{subsec:relhyp}.  Finally, in the Appendix (\Cref{sec:appendix}), we include a technical lemma about contracting directed geodesics in the thick part of Outer space, which is needed in \Cref{subsec:out}.
 
\subsection{Acknowledgments}
The authors thank Jeff Brock, Daniel Groves, and Yair Minsky for useful conversations and Cornelia Dru\c tu for several helpful comments. We are also grateful to Christopher Cashen for pointing out an error in an earlier version of \Cref{lem:contract} and for providing us with the correction.
The first, second, and third authors were partially supported by NSF grants DMS-1502623, DMS-1045119, and DMS-1400498, respectively.  The second and third author would also like to thank the Mathematical Sciences Research Institute for hosting them during the completion of this project. Finally, we thank the anonymous referee for helpful suggestions.

\section{Background} \label{sec:background}

Let $X$ be a geodesic metric space.  
Then $X$ is proper if closed balls are compact, and an action $G \curvearrowright X$ is proper if for each compact $K \subset X$, the set $\{g : gK\cap K \neq \emptyset \}$ is finite. By definition, a path $\gamma$ is a continuous map from an interval into $X$.
If $\gamma$ is any finite 
path in $X$, let $|\gamma|$ denote the distance between its endpoints, and let $\|\gamma\|$ denote its arclength. The \define{slope} of $\gamma$ is defined as
\[
\sl(\gamma) = \frac{\|\gamma\|}{|\gamma|}.
\]

For $X, Y$ metric spaces, $\lambda>0$ and $\kappa \ge 1$, a \emph{$(\kappa, \lambda)$-quasi-isometric embedding} of $X$ into $Y$ is a map $\phi:X \rightarrow Y$ so that for any $a,b \in X$, 
\[  
\frac{1}{\kappa} d_{X}(a,b)- \lambda \le d_{Y}(\phi(a), \phi(b)) \le \kappa \cdot d_{X}(a,b) + \lambda. \]
Finally, a \emph{$(\kappa, \lambda)$-quasigeodesic} in a metric space $Y$ 
is a $(\kappa, \lambda)$-quasi-isometric embedding of an interval into $Y$, where we allow for the possibility that the interval is infinitely long. By a $\kappa$--quasigeodesic, we mean a $(\kappa,\kappa)$--quasigeodesic.

For an injective map $i \colon X \to Y$ between metric spaces, we say that $X$ is \define{undistorted} in $Y$ if $i$ is a quasi-isometric embedding. In this case, we often identify $X$ with its image in $Y$. We will sometimes abuse our terminology and say that a subspace $X \subset Y$ is undistorted, but in this case the metric on $X$ will be clear from context. Similarly, we will often blur the distinction between a quasigeodesic and its image in $X$.

For $X$ a subspace of $Y$ and $D>0$, let $N_{D}(X)$ denote the $D$-neighborhood of $X$ in $Y$. Then given $X, X' \subseteq Y$ two subspaces of a metric space $Y$, we say that $X$ and $X'$ are \define{$D$--Hausdorff close} if $X \subset N_{D}(X')$ and $X' \subset N_{D}(X)$. 

\subsection{Stability and the sublinear contracting projections}
We begin by defining stability for subgroups of a finitely generated group $G$. We then recall the definition of sublinear contracting projections.




\begin{definition}[Stability] \label{def:stab}
Let $Y \subseteq X$ be an undistorted subspace of a metric space $X$. Then $Y$ is \define{stable} in $X$ if for any  $\kappa>1$, there exists $D>0$ so that if $q, q'$ are $\kappa$--quasigeodesics in $X$ with the same endpoints $a,b \in Y$, then $q$ and $q'$ are $D$--Hausdorff close. 
\end{definition}

A \define{stable embedding} of $Y$ into $X$ is a quasi-isometric embedding $\phi: Y \rightarrow X$ so that the image $\phi(Y)$ is stable in $X$. 

Although we have defined stability in a general setting, our focus will be the case of a finitely generated group $G$. Fix a finite generating set $S$ of $G$ and let $|\cdot|_S$ be the associated word norm. Recall that any two generating sets of $G$ give quasi-isometric metrics and that a finitely generated subgroup $H \le G$ is undistorted in $G$ when the inclusion $H \to G$ is a quasi-isometric embedding for some (any) word metrics on $H$ and $G$. 
Then a finitely generated subgroup $H \le G$ is \define{stable} if $H$ is undistorted in $G$ and $H \subset (G, |\cdot|_S)$ is stable for any choice of word metric on $H$. We note that for a pair of finitely generated groups $H \le G$, stability of $H$ in $G$ is independent of the generating sets for $H$ or $G$. Recent work on stable subgroups appears in \cite{koberda2014geometry, cordes2016stability, cordes2016boundary, cordeshume2}.\\

Our proof of \Cref{intro:th_pulling_back} uses a recent result of Arzhantseva--Cashen--Gruber--Hume \cite{ACGH} which characterizes stable quasigeodesics using a contracting property for the nearest point projection. We follow their discussion closely.

Let $X$ be a geodesic metric space and $Y\subset X$ a subspace. For $\epsilon >0$, the $\epsilon$-closest point projection $\pi_Y^\epsilon: X \rightarrow 2^Y$ maps a point $x \in X$ to the set of points of $Y$ whose distance to $x$ is within $\epsilon$ of $d(x,Y)$:
\[
\pi_Y^\epsilon(x) = \{y \in Y : d(x,y) \le d(x,Y) + \epsilon \}.
\]
Let $\rho$ be a sublinear function, i.e. a function which is nondecreasing, eventually nonnegative, and for which $\lim_{r \to \infty}\frac{\rho(r)}{r} = 0$. Then given $\epsilon>0$, $Y \subset X$ is \define{$(\rho, \epsilon)$--contracting} if 
for all $x,x' \in X$, $d(x,x')\le d(x,Y)$ implies that 
\[
\diam \left( \pi_Y^\epsilon(x) \bigcup \pi_Y^\epsilon(x') \right) \le \rho \big(d(x,Y) \big).
\]
We also say that $Y$ is \emph{sublinearly contracting} with \emph{contracting function} $\rho$. For our applications, $Y$ will always be the image of a quasigeodesic in $X$ and we call the resulting quasigeodesic sublinearly contracting. One of the key results of \cite{ACGH} is the following:

\begin{proposition}[Proposition 4.2 of \cite{ACGH}] \label{prop:stab_contract}
For each stability function $D$ and $\epsilon>0$ there is a sublinear function $\rho$ such that if $Y$ is a $D$-stable subspace of a geodesic metric space $X$, then $Y$ is $(\rho, \epsilon)$--contracting. 
\end{proposition}

\noindent Finally, we say a subspace $Y$ is \emph{$\rho$--contracting} if it is $(\rho,1)$--contracting.

We conclude this section with a useful fact that was pointed out to us by Christopher Cashen. If $\rho$ is a sublinear function, then by setting $\bar \rho(r) = r \cdot \sup_{s\ge r} \frac{\rho(s)}{s}$ we obtain a new sublinear function such that $\rho(r) \le \bar\rho(r)$ for all $r$ having the property that $\bar \rho(r)/r$ is nonincreasing. 

\section{Middle recurrence and stability} \label{sec:mr and stab}

For a path $\gamma$ in $X$, we sometimes write $a \in \gamma$ to mean $a \in \im(\gamma)$, the image of $\gamma$. For $a,b \in \gamma$, we say that $x\in \gamma$ \emph{lies between} $a,b$ if there are $s_1 \le s_2 \le s_3$ in the domain of $\gamma$ with $\gamma(s_1) =a, \gamma(s_2) =x, \gamma(s_3)=b$.
For $t\in (0,1/2)$, the \define{$t$--middle of $\gamma_{[a,b]}$} is the set of $x \in \gamma$ lying between $a$ and $b$ such that $\min \{d(x,a),d(x,b)\} \ge t\cdot d(a,b)$. We denote the $t$--middle of $\gamma_{[a,b]}$ by $\gamma_{\mathbf{t}[a,b]}$. If $\gamma$ is defined on a finite length interval, then its $t$--middle is denoted simply $\gamma_\mathbf{t}$.

\begin{figure}[htbp]
\begin{center}
\includegraphics[width = .5 \textwidth]{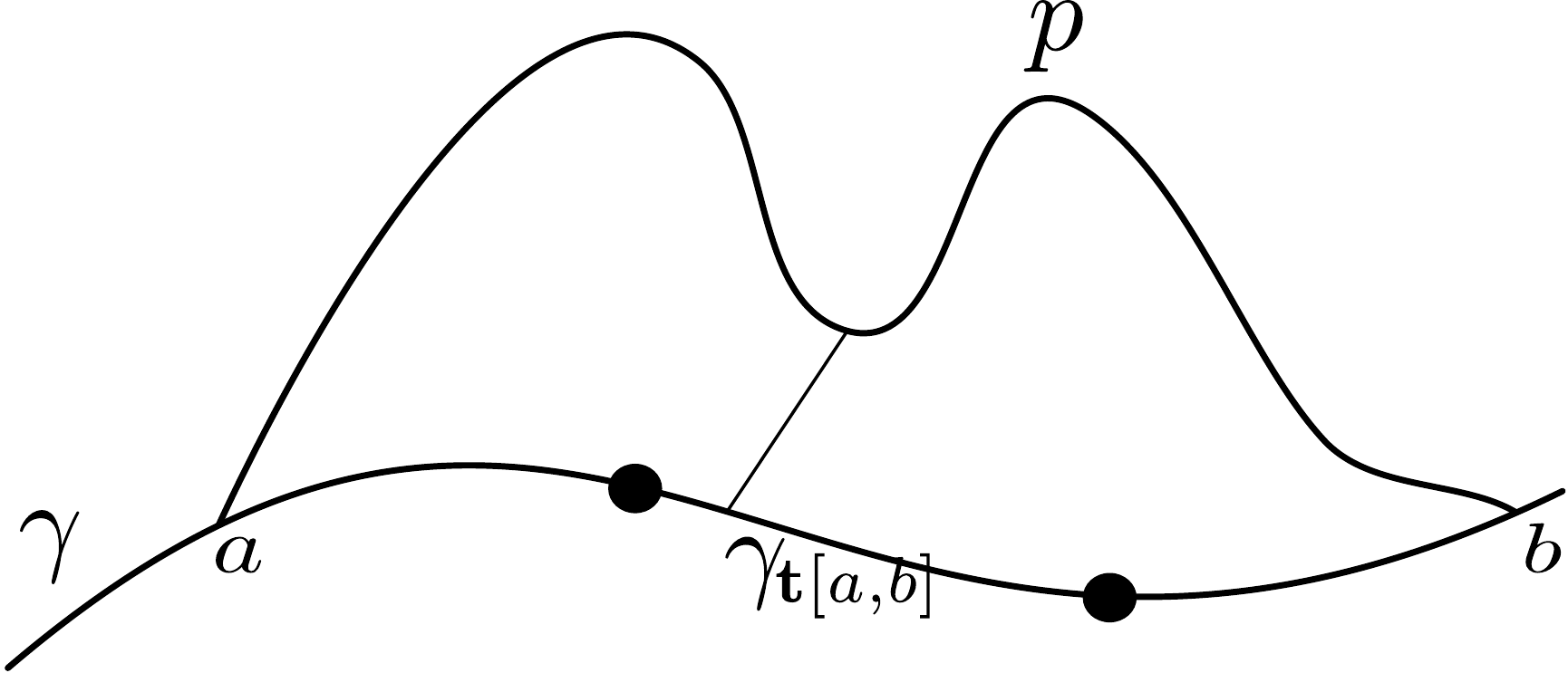}
\caption{The path $p$ returns to the $t$-middle of the path $\gamma$. In general, $\gamma_{\mathbf{t}[a,b]}$ need not be connected.}
\label{Fig:MR}
\end{center}
\end{figure}

We say that the path (or quasigeodesic) $\gamma$ is \define{t--middle recurrent} if there is a function $\m_{t,\gamma} \colon \R_+ \to \R_+$ so that for any path $p$ with endpoints $a,b \in \gamma$ satisfying $\|p\|\le C \cdot d(a,b)$, we have 
\[
p \cap N_{\m_{t,\gamma}(C)}(\gamma_{\mathbf{t[a,b]}}) \neq \emptyset.
\] 
The function $\m_{t,\gamma} \colon \R_+ \to \R_+$ is called the \define{$t$-recurrence function} of the path $\gamma$. The idea is simply that any path with controlled slope and endpoints on $\gamma$ must return to a bounded diameter neighborhood of the middle portion of $\gamma$. See \Cref{Fig:MR}.

When $\gamma$ is $t$--middle recurrent for each $t \in (0,1/2)$, we say that $\gamma$ is \define{middle recurrent}.

\subsection{Middle recurrent implies stable}

We now prove what is essentially one direction of the effective equivalence between the middle recurrence and stability conditions (\Cref{th:intro_middle_recurrence}):

\begin{theorem}
\label{th:MiddleRecurrence_stable}
Let $\gamma$ be a continuous path which satisfies $t$-middle recurrence for some $0<t < 1/2$ with recurrence function $\m_{t,\gamma}  \colon \R_+ \to \R_+$.  For any $\kappa \ge 1$, if $q$ is a $\kappa$-quasigeodesic with endpoints on $\gamma$, then 
$$q \subset N_{R}(\gamma),$$
for a constant $R \ge 0$ depending only on $\m_{t,\gamma}  \colon \R_+ \to \R_+$ and $\kappa$.
\end{theorem}


\begin{proof}
By replacing $q$ with a \emph{tame} quasigeodesic, as in \cite[Lemma III.H.1.11]{BridsonHaefliger:book}, we may assume that $q$ satisfies the inequality 
\[
||q|_{[a,b]}|| \le \kappa' d(q(a),q(b)) + \kappa',
\]
for all $a,b \in \mathrm{dom}(q)$ (the domain of $q$) and for some $\kappa'$ depending only on $\kappa$.

Set $d= \max(\m_{t,\gamma}(4\kappa'+1), 1)$.  We claim the following: 
if $[a',b']$ is a subinterval of the domain of $q$ such that $d(q(a'),\gamma),\: d(q(b'),\gamma)\leq d$ and $d(q(a'),q(b'))\ge 4 \cdot d/t$, then there is a point $c' \in [a',b']$ such that $d(q(c'), \gamma) \leq d$.

Before proving the claim, we show how it completes the proof.  Let $Q = \{x \in \mathrm{dom}(q) : d(q(x),\gamma)\leq d\}$. The claim implies that the complement of $Q$ in $\mathrm{dom}(q)$ is a collection of open intervals each of whose image in $X$ has length at most $4\kappa' \cdot d/t + \kappa'$. Hence, $q \subset N_{d+4\kappa'd/t +\kappa'}(\gamma)$, as required. 

To prove the claim, let $x,y \in \gamma$ such that $d(x,q(a')), d(y,q(b')) \leq d$. Let $\omega$ be the concatenation $ [x,q(a')] \cdot h \cdot [q(b'),y]$, where $h$ is the restriction of $q$ to the interval $[a',b']$. We compute the slope of $\omega$:

\begin{align*}
\frac{\| \omega \|}{d(x,y)} &\le \frac{\| h \| +2d}{d(q(a'),q(b'))-2d} \\
&\le \frac{\kappa' d(q(a'),q(b')) + \kappa' +2d}{1/2d(q(a'),q(b'))} \\
&\le \frac{\kappa' d(q(a'),q(b')) + \kappa' +1/2d(q(a'),q(b'))}{1/2d(q(a'),q(b'))}\\
&\le 4\kappa' + 1
\end{align*}

Thus, by middle recurrence, there is a point $c' \in \mathrm{dom}(\omega)$ and a point $z \in \gamma_{\mathbf{t}[x,y]}$ such that $d(\omega(c'),z)\le d$. Note that by definition of the $t$-middle, $d(z,x),d(z,y)\geq t\cdot d(x,y)$. To finish the claim, it suffices to show that $\omega(c')$ is in the image of $q$. Otherwise, $\omega(c')$ is contained in either $[x,q(a')]$ or $[q(b'),y]$. Suppose towards a contradiction that $\omega(c') \in [x,q(a')]$. Then
$$ t\cdot d(x,y) \le d(z,x) \le d(z,\omega(c')) + d(\omega(c'),x) \le 2d$$
and so $d(q(a'),q(b')) \le 2d/t  + 2d = 2d(1+1/t) <4 d/t$, a contradiction, completing the proof.
%
%
%
%
%
%
\end{proof}

\subsection{Stable implies middle recurrent}
Having established \Cref{th:MiddleRecurrence_stable}, it remains to prove that stable quasigeodesics are middle recurrent.

\begin{theorem}
\label{th:S_MT}
For a given $t\in (0,1/2)$, stability function $D \colon \mathbb{R}^{2}_{+} \rightarrow \mathbb{R}_{+}$, $\kappa \ge 1$, $\lambda \ge 0$, and Lipschitz constant $C \geq 1$, there exists $M= M(t,D,C, \kappa, \lambda)  \ge 0$ satisfying the following: If $\gamma$ is a stable $(\kappa, \lambda)$-quasigeodesic in a geodesic metric space $X$ with stability function $D$, and $p$ is a path with endpoints $a,b \in \gamma$ satisfying $\|p\|<C \cdot d(a,b)$, then
$$p \cap N_{M}(\gamma_{\mathbf{t[a,b]}}) \neq \emptyset.$$ 
\end{theorem}

The proof of \Cref{th:S_MT} will require the following lemma.

\begin{lemma} \label{lem:contract}
Assume that $\rho$ is a sublinear function such that $\rho(r)/r$ is nonincreasing, and
let $\gamma$ be a $\rho$--contracting quasigeodesic.  Suppose that $h$ is a path at distance at least $K\ge0$ from $\gamma$ whose endpoints have distance exactly $K$ from $\gamma$.
Then
\[ 
\frac{\rho(K)}{K} \ge \frac{1 - (2K+ \rho(K))/|h|}{\sl(h)}. 
\]
\end{lemma}

\begin{proof}
We begin by decomposing $h$ into a sequence of subpaths $h_1, \ldots h_m$ such that 
\begin{enumerate}
\item for $1\le i <m$, the initial endpoint $x_i$ of $h_i$ is such that $|h_i| = d(x_i,\gamma)$, and 
\item for $1\le i \le m$, $h_i$ is contained in the $d(x_i,\gamma)$--ball about $x_i$.
\end{enumerate}
Note that the terminal endpoint of $h_i$ is $x_{i+1}$ for $i <m$, and set $r_i = d(x_i,\gamma)$ for each $1 \le i \le m$. Since the terminal endpoint of $h$ has distance $K$ from $\gamma$, $r_m \le |h_m| + K$.

Such a decomposition of $h$ can be obtained inductively as follows: $h_1$ is the maximal initial subpath along the parameterization of $h$ contained in the closed ball of radius $r_1=K$ about $x_1$, the initial endpoint of $h$. Now suppose that $h_1, \ldots, h_{j-1}$ have been defined and let $x_{j}$ be the terminal endpoint of the path $h_{j-1}$. Set $h_{j}$ to be the maximal initial subpath of $h$ along the parameterization starting at $x_{j}$ which is contained in the closed ball of radius $r_j= d(x_j,\gamma)$ about $x_j$. This construction, along with continuity of $h$, gives a decomposition with the required properties. 

Using that $\gamma$ is $\rho$--contracting, we see that the diameter of the projection of $h_{i}$ to $\gamma$ is no more than $\rho(r_i)$.  So by concatenating a geodesic from the initial endpoint of $h$ to $\gamma$, the images of the projections of our subpaths, and a geodesic from $\gamma$ to the terminal endpoint of $h$, we obtain that
\[
|h| \le 2K + \sum_{i=1}^m \rho(r_i).
\]

As $r_i \ge K$ for each $i$ and $\rho(r)/r$ is nonincreasing, $\rho(r_i) \le r_i \cdot \rho(K)/K$. Hence the above inequality implies that 
\[
|h| \le 2K + \left (\rho(K) / K \right )\sum_{i=1}^m r_i .
\]
Now using that $|h_i| = r_i$ (for $i<m$), $r_m \le |h_m| + K$, and $\sum |h_i| \le \sum ||h_i|| = ||h||$, we get
\begin{align*}
|h| &\le 2K + \left (\rho(K) / K \right ) (||h||+K) \\
&= 2K + \rho(K) + \left (\rho(K) / K \right ) ||h|| \\
&= 2K + \rho(K) + \left (\rho(K) / K \right ) \sl(h) \cdot |h|.
\end{align*}
A rearrangement of the above yields
\[ \frac{\rho(K)}{K} \ge \frac{|h| - (2K+ \rho(K))}{|h| \cdot \sl(h)}, \]
and gives the required inequality.

%
%
\end{proof}

%
%
%
%
%
%
%

We now prove \Cref{th:S_MT}:

\begin{proof}
Fix $t \in (0, 1/2)$.
Let $\gamma \colon I \rightarrow X$ be a $D$-stable $(\kappa, \lambda)$-quasigeodesic. 
By \Cref{prop:stab_contract}, there is a function $\rho$, depending only on $D$, so that $\gamma$ is $\rho$--contracting. As in the paragraph following \Cref{prop:stab_contract}, we may assume $\rho$ is chosen so that $\rho(r)/r$ is nonincreasing.
Now fix $C \ge 1$ and let $L\ge0$ be such that there is a path $p$  with endpoints $a,b \in \gamma$ such that $\|p\| \le C |p|$ and 
$$p \cap N_{L}(\gamma_{\mathbf{t}[a,b]}) = \emptyset.$$
We prove that $L$ is uniformly bounded. Let $\gamma_{\mathbf{l}}$ (resp. $\gamma_{\mathbf{r}}$) be the set of all points $x \in \gamma$ so that either $x$ is to the left of $a$ in the orientation of $\gamma$ (resp. to the right of $b$), or $x \in \gamma_{[a,b]}$ and $d(x,a) < t \cdot d(a,b)$ (resp. $d(x,b) < t \cdot d(a,b)$). Note that we are free to assume that $d(a,b) = |p|> 4\lambda \kappa^2/(1-2t)$.

Since $\gamma_{\textbf{l}}$, $\gamma_{\textbf{r}}$, and $\gamma_{\mathbf{t}[a,b]}$ are pairwise disjoint by construction, 
\begin{equation} \label{union}
\gamma= \gamma_{\mathbf{l}} \sqcup \gamma_{\mathbf{t}[a,b]} \sqcup \gamma_{\mathbf{r}}. 
\end{equation}
We note that $\gamma_{\textbf{l}}$, $\gamma_{\textbf{r}}$, and $\gamma_{\mathbf{t}[a,b]}$ may each be disconnected, but this will not matter in what follows.

\begin{claim}\label{claim1}
For any $x,y \in I$ with $\gamma(x) \in \gamma_{\mathbf{l}}$ and $\gamma(y) \in \gamma_{\mathbf{r}}$, 
\begin{equation} \label{fatmiddle}
 d(\gamma(x),\gamma(y)) > \kappa^{-2} (1-2t)|p| - 2\lambda.  
 \end{equation}
\end{claim}
 
\begin{proof}[Proof of \cref{claim1}]
Assume first that $\gamma(x), \gamma(y) \in \gamma_{[a,b]}$. Then $d(\gamma(x),a), d(\gamma(y),b) < t \cdot|p|$, and the triangle inequality gives
$$d(a,b) < td(a,b) + d(\gamma(x), \gamma(y)) + td(a,b)$$
and a rearrangement gives that $d(\gamma(x),\gamma(y)) \geq (1-2t)d(a,b)$.

We next assume that $\gamma(x)$ is to the left of $a$ and $\gamma(y)$ is to the right of $b$; let $s,s' \in I$ so that $\gamma(s)=a$ and $\gamma(s') = b$.  Then $x < s< s'<y$ and since $\gamma$ is a $(\kappa, \lambda)$-quasigeodesic, we get
$$\kappa^{-1} (d(a,b) - \lambda) \leq s'-s\leq y-x \leq \kappa(d(\gamma(x),\gamma(y))+\lambda)$$
and a rearrangement and simplification gives $d(\gamma(x),\gamma(y)) \geq \kappa^{-2}d(a,b) - 2\lambda$.  

Finally, a similar argument proves that in the mixed case when (without loss of generality) $\gamma(x)$ is to the left of $a$ and $\gamma(y) \in \gamma_{[a,b]}$, 

\[ d(\gamma(x), \gamma(y)) \ge \kappa^{-2} d(a, \gamma(y))- 2\lambda. \]
Hence, the desired bound follows from the fact that $d(a, \gamma(y)) \ge (1-2t)\cdot |p|$. This proves the claim. 
\end{proof}

Now set 
\[K = \frac{1}{8} \left[\frac{(1-2t)L}{\kappa^2C}- 2\lambda \right],\]
 where we assume that $L$ is sufficiently large so that $K>0$ and that $\rho(K)/K \le 1$. (Recall that $\lim_{r \to \infty} \rho(r)/r = 0$.)
 Then since $L \le \|p\| \le C|p|$, we have that 
 \[K \le \frac{1}{8} \left[\kappa^{-2}(1-2t)|p|- 2\lambda \right]. \]
 Since $p$ avoids the $L$--neighborhood of the $t$-middle of $\gamma$, it also avoids the $K$--neighborhood of the $t$-middle of $\gamma$.

Let $\mathcal{H}$ be the collection of maximal subarcs of $p$ which lie outside of the open $K$--neighborhood of $\gamma$. Order $\mathcal{H}$ using the orientation of $p$. 

\begin{claim}\label{claim2}
There is an $h \in \mathcal{H}$ which has one endpoint that is $K$--close to $\gamma_{\mathbf{l}}$ and one endpoint $K$--close to $\gamma_{\mathbf{r}}$.
\end{claim}
 
\begin{proof}[Proof of \cref{claim2}]
First note that any point of $p$ contained in the closed $K$--neighborhood of $\gamma$, including endpoints of paths in $\mathcal{H}$, has distance at most $K$ from either $\gamma_{\mathbf{l}}$ or $\gamma_{\mathbf{r}}$. This follows from (\ref{union}) and the fact that the entire path $p$ avoids the $K$-neighborhood of the $t$--middle of $\gamma$. 

If each path in $\mathcal{H}$ has both endpoints close to only one of $\gamma_{\mathbf{l}}$ or $\gamma_{\mathbf{r}}$, then there must exist consecutive subpaths $h_i$ and $h_{i+1}$ in $\mathcal{H}$ such that the terminal endpoint of $h_i$ is close to a different end of $\gamma$ than the initial endpoint of $h_{i+1}$. If $q$ is the closed subpath of $p$ between $h_i$ and $h_{i+1}$, then each point of $q$ lies within distance $K$ from either $\gamma_{\mathbf{l}}$ or $\gamma_{\mathbf{r}}$, and its endpoints are $K$-close to different ends. Hence, by continuity, there is a point $z \in q$ which is $K$--close to points in both $\gamma_{\mathbf{l}}$ and $\gamma_{\mathbf{r}}$. This gives a point in $\gamma_{\mathbf{l}}$ and a point in $\gamma_{\mathbf{r}}$ which have distance at most $2K \le \frac{1}{4}[\kappa^{-2} (1-2t)|p|- 2\lambda]$ from one another. This, however, is impossible as it contradicts (\ref{fatmiddle}). 
\end{proof}

Hence, there is $h \in \mathcal{H}$ with endpoints $x,y$ such that $d(x,\gamma_{\mathbf{l}}) \le K$ and $d(y, \gamma_{\mathbf{r}})\le K$. Since by (\ref{fatmiddle}) any point in $\gamma_{\mathbf{l}}$ has distance at least $\kappa^{-2} (1-2t)|p| - 2\lambda$ from any point in $\gamma_{\mathbf{r}}$, we have
\begin{align*}
|h| &\ge \kappa^{-2} (1-2t)|p| - 2\lambda - 2K  \\
 &\ge \kappa^{-2} (1-2t)|p| - 2\lambda - \frac{1}{4}(\kappa^{-2}(1-2t)|p|- 2\lambda) \\
  &= \frac{3}{4}(\kappa^{-2}(1-2t)|p| -2\lambda), 
\end{align*}
and in particular $|h| \ge 6K$. Further using the assumption that that $|p|> 4\lambda \kappa^2/(1-2t)$, we have the estimate
\[
|h| \ge \frac{3}{8}(\kappa^{-2}(1-2t)|p|).
\]

Therefore, since $||p|| \le C|p|$
\[
\sl(h) = \frac{\|h\|}{|h|} \le \frac{\|p\|}{\frac{3}{8}(\kappa^{-2}(1-2t)|p|)} \le  \frac{8C|p|}{3\kappa^{-2}(1-2t)|p|} = \frac{8C}{3\kappa^{-2}(1-2t)}.
\]

Finally, applying \Cref{lem:contract} to $h$ yields
\begin{align*}
\frac{\rho(K)}{K} &\ge \frac{1 - (2K+ \rho(K))/|h|}{\sl(h)} \\
&\ge  \frac{1- 3K/|h|}{\sl(h)}\\
&\ge  \frac{1/2}{\sl(h)}\\
&\ge  \frac{3\kappa^{-2}(1-2t)}{16C} >0
\end{align*}
where the second and third inequalities hold since $\rho(K)/K \le 1$ and $|h| \ge 6 K$, respectively.
As $\rho$ is sublinear, $\lim_{r \to \infty} \rho(r)/r = 0$. Therefore, the quantity $K$, and hence $L$, is uniformly bounded above in terms of the constants $C,t, \kappa, \lambda$, and the function $\rho$.  This completes the proof.
\end{proof}

We conclude this section by putting together \Cref{th:MiddleRecurrence_stable} and \Cref{th:S_MT} to prove \Cref{th:intro_middle_recurrence}.

\begin{proof}[Proof of \Cref{th:intro_middle_recurrence}]
Let $q$ be a middle recurrent quasigeodesic in $X$. As in the proof of \Cref{th:MiddleRecurrence_stable}, we first note that there is a tame (continuous) quasigeodesic $q'$ which lies at uniformly bounded Hausdorff distance from $q$. Hence, $q'$ is also middle recurrent, with recurrence function depending only on that of $q$. By \Cref{th:MiddleRecurrence_stable}, applied to the path $q'$, we see that $q'$, and hence $q$, is a stable quasigeodesic whose stability function depends only on the recurrence function of $q$.
The converse is exactly the statement of \Cref{th:S_MT}.
\end{proof}

\subsection{Middle recurrence in the literature} \label{sec:Mid_Lit}
In this subsection, we identify and address an error in the literature.

The notion of middle recurrence for a quasigeodesic $\gamma$ in $X$---the central concept of this project---was first introduced by Drutu--Mozes--Sapir \cite{DMS}.  Proposition 3.24 of \cite{DMS} gives a list of $5$ conditions which are claimed to be equivalent to stability of a biinfinite quasigeodesic. 
One of these conditions is middle recurrence. We note that since their proof runs through an argument involving asymptotic cones, it does not give an explicit relationship between the stability and recurrence functions as is achieved by \Cref{th:intro_middle_recurrence} and which is necessary for \Cref{th:pull_back}. 

We now show that one of these conditions --which we refer to as \define{Property 5}  as it is the fifth listed condition -- is in fact strictly stronger than stability: 

\begin{itemize}
\item []
\define{Property 5}: For every $C\ge 1$, there exists $K \geq 0$ so that if $p$ is a path with endpoints $a,b$ on $\gamma$ satisfying $||p|| \leq C \cdot d(a,b)$, the portion $\gamma|_{[a,b]}$ of $\gamma$ between $a$ and $b$ lies within $K$ of $p$. 
\end{itemize}

The following example demonstrates that geodesics in the hyperbolic plane do not satisfy Property $5$; however, they are all stable, as is any geodesic in a Gromov hyperbolic space. 

\begin{example} Let  $\gamma$  be the horizontal diameter in the disk model of $\mathbb{H}^{2}$. Let $S_D$ be the sphere of radius $D$ about the origin (which $\gamma$ runs through; identify $\gamma$ with $\mathbb{R}$ in the obvious way, identifying the center of the disk with the origin). Note that the upper hemisphere of $S_D$ has length on the order of $e^D$. Let $p$ be the following path: let $a$ be the point corresponding to $-e^{D}$ on $\gamma$ and follow $\gamma$ until reaching $-D$; then follow along the upper hemisphere of $S_D$ until arriving at $D$; then proceed along $\gamma$ until reaching $e^D$, which we call $b$. Then $\|p \| \le (\pi+2)e^D \le 3 \cdot d(a,b)$. Note that the origin, which is a point on $\gamma$, does not lie in a $D$--neighborhood of $p$ by construction. Since $D$ can be made arbitrarily large, $\gamma$ does not satisfy Property $5$. 
\end{example}

We note that as $D \rightarrow \infty$, the diameter of $S_{D}$ is vanishingly small compared to the full length of $\gamma|_{[a,b]}$, and therefore this example does not contradict the middle recurrence property. 

The proof in \cite{DMS} that stability implies middle recurrence factors through Property $5$.  However,  \Cref{th:intro_middle_recurrence} gives a direct proof of the equivalence of middle recurrence and stability.

\section{Pulling back stability} \label{sec:pbs}

Using the middle recurrence property for stable quasigeodesics we prove our main theorem (\Cref{intro:th_pulling_back}). We encourage the reader to observe that its proof relies on the uniform equivalence of stability and middle recurrence as established in the previous section.

\begin{theorem}\label{th:pulling_back} \label{th:pull_back}
Let $G$ be a finitely generated group with a proper action $G \curvearrowright X$ on the proper geodesic metric space $X$. Let $H \le G$ be a 
subgroup such that for some $x \in X$, the orbit map $\orb \colon G \to X$ given by $g \mapsto gx$ restricts to a stable embedding on $H$. Then $H$ is stable in $G$.
\end{theorem}

\begin{proof} 
Let $S$ be a finite generating set for $G$ chosen so that $S\cap H$ generates $H$.
Denote the corresponding Cayley graph by $\Gamma(G,S)$, and extend the orbit map to
$\orb \colon \Gamma(G,S) \to X$ in the usual way. Note that $\Gamma(H, S\cap H)$ is a subgraph of $\Gamma(G,S)$.
Since $\orb \colon \Gamma(G,S) \to X$ is coarsely Lipschitz and its restriction to $H$ is a quasi-isometric embedding, it follows easily that $H \le G$ is undistorted. It suffices to show that for any geodesic $\gamma$ in $\Gamma(H,S\cap H)$, $\gamma$ is $f$--stable as a quasigeodesic in $\Gamma(G,S)$, where the stability function $f$ is independent of $\gamma$. By \Cref{th:MiddleRecurrence_stable}, this is equivalent to the statement that there is a constant $0<t<1/2$ such that for each $C\ge 0$ there is an $M \ge0$ satisfying the following: if $p$ is any path in $G$ with $\|p \| \le C |p|$ sharing endpoints with $\gamma$, then
\[
p \cap N_{M}(\gamma_{\mathbf{t}}) \neq \emptyset.
\]

To this end, let $C\ge0$ and let $p$ be such a path. We use $\widebar{\cdot}$ notation for the images of objects under the orbit map $\orb$.
Since $\orb \colon \Gamma(H,H\cap S) \to X$ is a stable embedding, the image $\widebar{\gamma}$ is an $f_X$--stable quasigeodesic in $X$, for a stability function $f_X$ not depending on $\gamma$. Moreover, $\widebar{p}$ is a path which shares endpoints with $\widebar{\gamma}$ such that $\|\widebar p\| \le C_X |\widebar{p}|$, where $C_X \ge 0$ depends only on $C$ and $S,X$.
Hence, by \Cref{th:S_MT}, there is an $M_X$, again depending only on $C$ and  $S,X$, such that 
\[
\widebar{p} \cap N_{M_X}(\widebar{\gamma}_{\mathbf{1/3}}) \neq \emptyset.
\]

Since $G$ acts properly on $X$, there exists $M>0$ depending only on $M_X$ so that if $d_{X}(g_{1} x, g_{2} x) \leq M_X$, then $d_{G}(g_{1}, g_{2}) \leq M$. Hence, there are $x \in p$ and $z \in \gamma$ such that $\widebar{z} \in \widebar{\gamma}_{\mathbf{1/3}}$ and $d_G(x,z) \le M$,  so it only remains to show that $z\in \gamma_{\mathbf{t}}$ for $0 < t< 1/2$ not depending on $\gamma$. However, this follows easily from the fact that $\orb \colon H \to X$ is a quasi-isometric embedding.
\end{proof}

\section{Applications} \label{sec:app}


\subsection{Stability in the mapping class group} \label{subsec:mod}

We begin with a shorter, simpler proof of one direction of the main theorem from \cite{DT15}:

\begin{theorem}[\cite{DT15}] \label{thm:stab_map}
Let $H$ be a convex cocompact subgroup of the mapping class group. Then $H \le \MCG(S)$ is stable.
\end{theorem}

\begin{proof}
Since $H$ is convex cocompact, for any $X \in \T(S)$, the orbit $H \cdot X$ is quasiconvex for the Teichm\"uller metric. This implies that all Teichm\"uller geodesics joining orbit points in $H \cdot X$ are uniformly thick, and so the orbit map $H \to \T(S)$ is stable by \cite{minsky1996quasi}. See also \cite[Theorem 6.3]{KentLein}.

Since $\MCG(S) \curvearrowright \T(S)$ satisfies the hypothesis of \Cref{th:pulling_back}, we conclude that $H$ is stable in $\MCG(S)$ as required. 
\end{proof}

\subsection{Stability in $\mbox{Out}(F_{n})$}\label{subsec:out}
The free factor complex $\F_n$ is the simplicial complex whose vertices are proper conjugacy classes of free factors of $F_n$ and whose $k$--simplices correspond to chains $A_0 \le A_1 \le \ldots \le A_k$. The action of $\Out(F_n)$ on conjugacy classes of free factors extends to a simplicial action of $\F_n$, which like the action of the mapping class group on the curve complex, is highly nonproper. However, Bestvina and Feighn proved that $\F_n$ is hyperbolic \cite{BF14}. 

The main result of this section is the following:

\begin{theorem} \label{th: Out}
Let $H$ be a finitely generated subgroup of $\Out(F_n)$ which has a quasi-isometric orbit map into the free factor complex $\F_n$. Then $H$ is stable in $\Out(F_n)$.
\end{theorem}

Roughly the same argument as in \Cref{thm:stab_map} can be used along with \Cref{th:pulling_back}  to show that convex cocompact subgroups of $\Out(F_n)$ are stable. The only difficulty in applying \Cref{th:pulling_back} in this setting is the fact that the Lipschitz metric on Outer space is asymmetric. Hence, we must appeal to the arguments in \Cref{th:pulling_back} rather than the theorem statement itself. 

For the remainder of this subsection, we assume that the reader has some familiarity with the Lipschitz metric on Outer space. See for example \cite{FMout, BF14}.

Let $\X_n$ denote Outer space with the Lipschitz metric $d$. Informally, $d(G_1,G_2)$ is the logarithm of the minimal Lipschitz constant among all maps from $G_1$ to $G_2$ in the correct homotopy class. We recall that in general $d(G_1,G_2) \neq d(G_2,G_1)$ and we set $\symd(G_1,G_2) = d(G_1,G_2) + d(G_2,G_1)$. Although the \emph{symmetrized metric} $\symd$ is an honest metric on $\X_n$ which induces the usual topology, $(\X_n,\symd)$ is not a geodesic metric space \cite{FMout}. We remark, however, that the metric space $(\X_n,\symd)$ is proper and that the natural action $\Out(F_n) \curvearrowright \X_n$ is properly discontinuous. Finally, we denote by $\X^{\ge \epsilon}_n$ the $\epsilon$--thick part of $\X_n$; this is the subspace of $\X_n$ consisting of graphs whose shortest essential loop has length at least $\epsilon$.

Recall that a directed geodesic $\gamma \colon I \to \X$ is $D$--strongly contracting if for any $H,H' \in \X$ with $d(H,H') \le d(H,\gamma)$, the diameter of the projection of the geodesic from  $H$ to $H'$ to $\gamma$ is bounded by $D$. See \cite{algom2011strongly, dowdall2015contracting}.

We need the following lemma which gives the required middle recurrence statement for strongly contracting geodesics in $\X_n$, analogous to \Cref{lem:contract}. The proof is provided in the Appendix (\Cref{sec:appendix}).

\begin{lemma}\label{lem:Out_middle}
For each $D, C \ge0$ there is an $R\ge0$ such that the following holds:
Let $\gamma \colon [0,T] \to \X^{\ge \epsilon}_n$ be a directed geodesic 
which is $D$--strongly contracting. Suppose that $p$ is a path in $\X^{\ge \epsilon}_n$ with endpoints $a = \gamma(0)$ and $b = \gamma(T)$  such that $\|p\| \le C \cdot d(a,b)$. Then there are $x\in p$ and $t \in \mathbb{R}$ with $T/3 \le t \le 2T/3$ such that $\symd(x,\gamma(t)) \le R$.
\end{lemma}

We can now show that convex cocompact subgroups of $\Out(F_n)$ are stable.

\begin{proof}[Proof of \Cref{th: Out}]
Fix an orbit map $\Out(F_n) \to \X_n$. We note that for some $Q\ge0$, this map is $Q$--Lipschitz  and proper with respect to the metric $\symd$. By our assumption, the Lipschitz map $\X_n \to \F_n$ induces an orbit map $\Out(F_n) \to \F_n$ such that the restriction $H \to \F_n$ is a quasi-isometric embedding. At the cost of increasing $Q$, if necessary, suppose that the orbit map of $H$ into both $\X_n$ and $\F_n$ is a $Q$--quasi-isometric embedding.

Since the orbit map $\Out(F_n) \to \F_n$ is coarsely Lipschitz, it follows immediately that the subgroup $H \le \Out(F_n)$ is undistorted.  
As in the proof of \Cref{th:pull_back}, it suffices to show the following:
There is a constant $0<m<1/2$ such that for any $C\ge0$ there is an $R\ge0$ with the property that for any $a,b \in H$ and any path $p$ in $\Out(F_n)$ between $a,b$ with $\|p\|\le C\cdot d(a,b)$, the path $p$ meets the $R$--neighborhood of the $m$--middle of $[a,b]_H$. We remark that it suffices to assume that $d(a,b)$ is larger than some constant depending only on $H$.

To this end, let $\alpha$ denote the geodesic in $H$ between $a,b \in H$ and let $p$ be a path from $a$ to $b$ as above. As before, we use bar notation, $\widebar{\cdot}$, to denote the images of these objects under our fixed orbit map $\Out(F_n) \to \X_n$. By \cite[Theorem 1.6]{dowdall2015contracting}, the orbit $\widebar H$ is \emph{strongly contracting}, meaning that there are constants $D,\epsilon \ge0$ such that any directed geodesic in $\X$ joining points in $\widebar H$ is contained in $\X_n^{\ge \epsilon}$ and is $D$--strongly contracting. Hence, if we let $\gamma \colon [0,T] \to \X_n^{\ge\epsilon}$ denote the directed geodesic from $\widebar a$ to $\widebar b$, then \Cref{lem:Out_middle} applies to the path $\widebar p$.
In particular, we have that  $\| \widebar p\| \le 2Q \|p\| \le 2QC d(a,b)$, 
and so by \Cref{lem:Out_middle} there is a constant $K$, depending only on $H \le \Out(F_n)$ and our orbit map, such that for some $ x \in p$, we have $\symd(\widebar x, \gamma(t))\le K$ where $T/3\le t \le 2T/3$. Note that $T = d(\widebar a, \widebar b) \le \symd(\widebar a, \widebar b)$. 

By \cite[Theorem 4.1]{dowdall2014hyperbolic}, the paths $\gamma$ and $\widebar \alpha$ uniformly fellow travel in $\X_n$ with respect to the symmetric metric. Hence, at the cost of increasing $K$, we have that there is $z \in \alpha$ such that $\symd(\gamma(t), \widebar z) \le K$. Combining these facts, $\symd(\widebar x, \widebar z) \le 2K$ and since $\Out(F_n) \to \X_n$ is proper with respect to the symmetric metric, we conclude $d(x,z) \le R$ for some constant $R\ge0$ depending only on $K$ and the orbit map. Since $x\in p$ and $z \in \alpha$, it only remains to show that $z$  lies in the $m$--middle of $\alpha$, for some $m$ depending only on $H$ and the orbit map. For this, we compute

\begin{align*}
d(a,z) &\ge \frac{1}{Q}\symd(\widebar a, \widebar z) - 1
\ge \frac{1}{Q}(d(\widebar a, \gamma(t)) - K) -1\\
&\ge \frac{1}{Q}(1/3 \cdot d(\widebar a, \widebar b) -K) -1 \\
&\ge \frac{1}{Q}(1/3Q \cdot d(a,b) -Q/3 -K) -1 \\
&\ge \frac{1}{6Q^2}d(a,b)
\end{align*}
where the last inequality holds for $d(a,b)$ sufficiently large, depending only on $K$. Since a similar inequality holds to give a lower bound on $d(z,b)$, we conclude that $z$ is in the $m$--middle of $\alpha$ for $m = \frac{1}{6Q^2}$. This completes the proof. 
\end{proof}

\subsection{Relatively hyperbolic groups} \label{subsec:relhyp}
In this section, we characterize stable subgroups of groups which are hyperbolic relative to subgroups of linear divergence. 
We assume the reader has some familiarity with relatively hyperbolic groups \cite{Gromov:essay, farb1998relatively,bowditch2012relatively}, 
and in particular with the definition of relative hyperbolicity due to Groves-Manning \cite{GM08}.
See \cite{Hruska12} for a thorough treatment of the various equivalent definitions. 

Recall that for a pair $(G, \P)$ with $\P = \{P_1, \ldots P_n \}$ and $P_i \le G$,  the group $G$ is hyperbolic relatively to $\P$ if and only if the associated cusped space $\cusp(G,\P)$ is hyperbolic \cite{GM08}. Briefly, fix a generating set $S$ for $G$ which intersects each $P_i$ in a generating set for $P_i$; hereafter, $S$ will be implicit in the discussion. The \define{cusped space} $\cusp(G,\P)$ is the space obtained by attaching combinatorial horoballs along translates of each $P_i$ in the Cayley graph of $G$ -- see \cite{GM08} for a precise definition. We remark that since horoballs are attached equivariantly, there is an induced properly discontinuous action $G \curvearrowright \cusp(G,\P)$, and $\cusp(G,\P)$ itself is a proper metric space. The subgroups $P \in \P$ are called the peripheral subgroups.

Following \cite{farb1998relatively}, we let $\cone(G,\P)$ be the \define{coned-off} Cayley graph. This is the space obtained from the Cayley graph of $G$ by adding edges between all elements of $gP$ for $g \in G$ and $P \in \P$. As before, we have an action $G \curvearrowright \cone(G,\P)$, which, as opposed to the action on $\cusp(G,\P)$, is not proper. We remark that there is an equivariant, $1$--Lipschitz map $\cusp(G,\P) \to \cone(G,\P)$ which sends each vertex of each horoball to the vertex of $gP$ it lies over. \\


Both $\cusp(G,\P)$ and $\cone(G,\P)$ can be used to characterize stable subgroups of $G$, when $(G, \P)$ is relatively hyperbolic and each $P \in \P$ is one-ended and has linear divergence.

\begin{theorem}[Stability in relatively hyperbolic groups] \label{th:relhyp_char_stab}
Let $(G, \P)$ be relatively hyperbolic and suppose that each $P \in \P$ is one-ended and has linear divergence. Then the following are equivalent for a finitely generated subgroup $H \le G$:
\begin{enumerate}
\item $H$ is stable in $G$.
\item $H$ has a quasi-isometric orbit map into $\cusp(G,\P)$.
\item $H$ has a quasi-isometric orbit map into $\cone(G,\P)$.
\end{enumerate}
\end{theorem}

We note that the hypothesis that each $P_i$ has linear divergence cannot be removed. For example, if $G$ is additionally assumed to be hyperbolic, then any cyclic subgroup of $G$ is stable, regardless of whether it is contained in some $P_i$.  

The following proposition imposes no conditions on the peripheral subgroups:

\begin{proposition}\label{prop:pb_relhyp}
Suppose that $(G,\P)$ is relatively hyperbolic and that $H$ is a finitely generated subgroup $G$. If $H$ has an orbit map into $\cusp(G,\P)$ which is a quasi-isometric embedding, then $H$ is stable in $G$.
\end{proposition}

\begin{proof}
Since $\cusp(G,\P)$ is a locally finite, hyperbolic graph and the action of $G$ on $\cusp(G,\P)$ is proper, \Cref{th:pull_back} implies that $H$ is stable in $G$, as required.
\end{proof}

Using the $1$--Lipschitz, equivariant map $\cusp(G,\P) \to \cone(G,\P)$, we immediately obtain the implications $[(3) \implies (2) \implies (1)$] in \Cref{th:relhyp_char_stab} as a corollary to \Cref{prop:pb_relhyp}. Note that these implications hold without any conditions on the peripheral subgroups.

We remark that the condition of $H$ quasi-isometrically embedding into $\cone(G, \P)$ is equivalent to ``strong relative quasiconvexity", and hence \Cref{prop:pb_relhyp} gives another proof that such subgroups are hyperbolic; see \cite[Section 9]{Hruska12}.

The remainder of the section will prove the implication $[(1) \implies (3)]$ of \Cref{th:relhyp_char_stab}.

\subsubsection*{Peripheral subgroups and projections}
Fix a relatively hyperbolic group $(G, \PP)$, a generating set $S$ as above, and let $d_G$ denote distance in $G$ with respect to $S$. Let $\overline \P = \{gP : g\in G \text{ and } P_i \in \P\}$ be the set of all left translates of peripheral subgroups.
For each $x \in G$ and $P \in \overline \P$, let $\pi_P(x)$ be all $y \in P$ such that $d_{G}(x,y)\leq d_{G}(x,P)+1$; $\pi_P(x)$ is called the \emph{almost projection} of $x$ onto $P$, following \cite[Definition 4.9]{DS05}. 

For a subset $H$ of $G$ and $P \in \overline \P$, we denote by $\diam_P(H)$ the diameter of the set $\{\pi_P(x) : x \in H\}$. For $x,y \in G$, we also set $d_P(x,y) = \diam_P(\{x,y\})$.
%


%


\begin{lemma}\label{cor:srqc via proj}
Let $(G, \PP)$ be relatively hyperbolic and $H\le G$ be a finitely generated subgroup.  Any orbit map of $H$ into $\cone(G ,\P)$ is a quasi-isometric embedding if and only if $H$ is undistorted in $G$ and there exists $M_0>0$ such that for each $P \in \overline \P$, we have
$$\diam_{P}(H) < M_0.$$
\end{lemma}

\begin{proof}
The reverse implication is a consequence of \cite[Theorem 0.1]{Sisto:relhyp}, which establishes a formula for distance in $G$ in terms of distance in $\cone(G, \P)$ and projections to subgroups in $\overline{\P}$.
In our setting, this implies that when $\diam_{P}(H) < M_0$ for each $P \in \overline \P$, distance in $G$ (and hence $H$ by nondistortion) is coarsely equal to the distance between the corresponding orbit points in $\cone(G, \P)$.

For the forward direction, let $x, y \in H$ and let $\gamma$ be any geodesic in $H$ between $x$ and $y$; since $H$ is undistorted in $G$, there exists $\kappa \geq 1, \lambda \geq 0$ such that $\gamma$ is a $(\kappa,\lambda)$-quasigeodesic in $G$.
By \cite[Lemma 1.15]{Sisto:relhyp} (see also \cite[Lemma 4.15]{DS05}), there exist $M, R>0$ depending only on $G$ and $\kappa,\lambda$ so that if $d_{g\cdot P}(x, y) \ge M$ for some $g \in G, P \in \PP$, then there exist $x_P , y_P \in \gamma$ such that $x_P \in \mathcal{N}_{R}(\pi_{g \cdot P}(x))$ and $y_P \in \mathcal{N}_{R}(\pi_{g \cdot P}(y))$.
Hence $d_{H}(x_P,y_P)$ is coarsely bounded below by $d_{g\cdot P}(x, y)$.  However, $d_{\cone(G,\P)}(x_P,y_P) \leq 2R + 2$ by construction of $\cone(G,\P)$.

Thus if there is no bound on the projections of $H$ to the peripheral subgroups and their translates, then we obtain a contradiction of the fact that $H$ quasi-isometrically embeds in  $\cone(G,\P)$.
\end{proof}

\begin{proof}[Proof of \Cref{th:relhyp_char_stab}]

By our observations above, it suffices to prove the implication $[(1) \implies (3)]$. Hence, assume that $H$ is stable in $G$. By \Cref{cor:srqc via proj}, it suffice to show that there is a uniform bound on $\diam_P(H)$ for each $P \in \overline \P$. 

Towards a contradiction, assume that there exists sequences $h_i,h'_i \in H$ and $g_i \in G$ such that 
\[
d_{g_iP}(h_i,h'_i) \ge i
\]
for some peripheral subgroup $P \in \P$.

Let $p_i$ be a geodesic in $g_iP$ (with its induced word metric) joining points of $\pi_{g_iP}(h_i)$ and $\pi_{g_iP}(h'_i)$ at distance $i$. Since $H$ is undistorted in $G$, any $H$--geodesic $[h_i,h'_i]_H$ is a uniform quasigeodesic in $G$. Fix such an $H$--geodesic. Again by \cite[Lemma 1.15]{Sisto:relhyp}, there is a constant $R \ge0$ depending only on $H \le G$ such that there exist $x_i, y_i \in [h_i,h'_i]_H$ such that $d_G(x_i,\pi_{g_iP}(h_i)),d_G(y_i,\pi_{g_iP}(h'_i))<R$.  Since $p_i$, as a path in $G$, is also quasigeodesic with uniform constants, stability of $H \le G$ implies that $p_i$ and the portion of $[h_i,h'_i]_H$ between $x_i$ and $y_i$ uniformly fellow travel. This implies that the paths $p_i$ are also uniformly stable in $G$ and, again using that each $P \in \P$ are undistorted, that the $p_i$ are uniformly stable as geodesics in $g_iP$. Denote their common stability function by $D \colon \R_+ \to \R_+$.

Hence, for each $i \ge0$ we have a geodesic path $p_i$ in $ g_i P$ of length $i$ such that $p_i$ is $D$--stable, for a stability function that depends only on $H \le G$. Then $\gamma_i = g_i^{-1} p_i$ is a geodesic in $P$ of length $i$ which is also $D$--stable, since $g_i \in G$ induces an isometry between $P$ and $g_i P$. After left multiplication by an element of $P$, we may assume that each $\gamma_i$ has $1 \in P$ as an approximate midpoint and, after passing to a subsequence, conclude that there is a biinfinite geodesic $\gamma$ in $P$ such that $\gamma_i$ converges uniformly on compact sets to $\gamma$. We conclude that $\gamma$ is a biinfinite $D$-stable geodesic in the group $P$. By \cite[Lemma 3.15]{DMS}, the existence of a biinfinite stable quasigeodesic contradicts that $P$ has linear divergence and completes the proof.
\end{proof}


\section{Appendix} \label{sec:appendix}
In this section, we prove the technical \Cref{lem:Out_middle}. As before, we expect the reader to be somewhat familiar with the geometry of Outer space. 
We let $N^\leftarrow_k(\gamma)$ denote the \emph{inward} $k$-neighborhood of $\gamma$, that is all points whose distance to $\gamma$ no more than $k$. Recall that if $\gamma \subset \X^{\ge \epsilon}$, then $N^\leftarrow_k(\gamma) \subset \X^{\ge \epsilon'}$ for $\epsilon'$ depending only on $\epsilon$ and $k$.
Finally, we remind the reader that given $\epsilon >0$ there is an $M_\epsilon$ such that if $a,b \in  \X^{\ge \epsilon}$, then $d(a,b) \le M_\epsilon d(b,a)$ \cite{algom2012asymmetry}.

\begin{lemma}\label{lem:Out_middle}
For each $D, C \ge0$ there is an $R\ge0$ such that the following holds:
Let $\gamma \colon [0,T] \to \X^{\ge \epsilon}_n$ be a directed geodesic 
which is $D$--strongly contracting. Suppose that $p$ is a path in $\X^{\ge \epsilon}_n$ with endpoints $a = \gamma(0)$ and $b = \gamma(T)$  such that $\|p\| \le C \cdot d(a,b)$. Then there are $x\in p$ and $t \in \mathbb{R}$ with $T/3 \le t \le 2T/3$ such that $\symd(x,\gamma(t)) \le R$.
\end{lemma}

\begin{proof}
Fix $C,D \ge0$. We note that it suffices to prove that lemma for $T\ge 6D$.

Let $L\ge0$ be a number such that there is a path $p$  with endpoints $a = \gamma(0)$ and $b=\gamma(T)$, which avoids the inward $L$--neighborhood of $ \widebar \gamma = \gamma|_{[T/3, 2T/3]}$ and for which $\|p\| \le C d(a,b)$. We show that $L$ is uniformly bounded. Denote the portions of $\gamma$ which come before and after $\widebar \gamma$, by $\gamma_{\mathbf{l}}$ and $\gamma_{\mathbf{r}}$, respectively.
Note that $T = d(a,b) = |p|$ and that $T \le \|p\|  \le C d(a,b) = CT$

First, set $K = \frac{L}{4C(1+M_\epsilon)}$. Since $L \le \|p\| \le C |p| = CT$, we have that $K \le \frac{T}{4(1+M_\epsilon)}$. Since $p$ avoids the inner $L$--neighborhood of $\widebar \gamma$, it also avoids the $K$--neighborhood of $\widebar \gamma$. 

Let $\mathcal{H}$ be the collection of maximal subarcs of $p$ which lie outside of the open inner $K$--neighborhood of $\gamma$. Order $\mathcal{H}$ using the orientation of $p$. We claim that there is an $h \in \mathcal{H}$ which has one endpoint $K$--close to $\gamma_{\mathbf{l}}$ and one endpoint $K$--close to $\gamma_{\mathbf{r}}$. To see this, first note that any point of $p$ in the closed $K$--neighborhood of $\gamma$, include endpoints of paths in $\mathcal{H}$, has distance at most $K$ from either $\gamma_{\mathbf{l}}$ or $\gamma_{\mathbf{r}}$. This is because the entire path $p$ avoids $\widebar \gamma$. If each path in $\mathcal{H}$ has both its endpoints close to only one of $\gamma_{\mathbf{l}}$ or $\gamma_{\mathbf{r}}$, then there must consecutive subpaths $h_i$ and $h_{i+1}$ in $\mathcal{H}$ such that the terminal endpoint of $h_i$ is close to a different `end' of $\gamma$ than the initial endpoint of  $h_{i+1}$. If $q$ is the closed subpath of $p$ between $h_i$ and $h_{i+1}$, then each point of $q$ lies within distance $K$ from either $\gamma_{\mathbf{l}}$ or $\gamma_{\mathbf{r}}$, and its endpoint $K$--are close to different ends. Hence, by continuity, there is a point $t \in q$ which is $K$--close to points in both $\gamma_{\mathbf{l}}$ and $\gamma_{\mathbf{r}}$. This gives a point in $\gamma_{\mathbf{l}}$ and a point in $\gamma_{\mathbf{r}}$ which have distance at most $K (1+M_\epsilon)$ from one another (in the direction along $\gamma$). However, by construction, such points necessarily have oriented distance at least $T/3$ and so we obtain that $T/3 \le K (1+M_\epsilon) \le T/8$, a contradiction.

We conclude that there is a $h \in \mathcal{H}$ with endpoints $x,y$ such that $d(x,\gamma_{\mathbf{l}}) \le K$ and $d(y, \gamma_{\mathbf{r}})\le K$. Let $x' \in \gamma$ and $y' \in \gamma$ be points realizing these distances. We note that these points are within (symmetric) distance $D$ from the projections to $\gamma$ of $x$ and $y$, respectively.

Break up $h$ into $m$ consecutive subpaths $h_{1},..., h_{m}$ so that  $\|h_{i}\| = K$ for all $i < m$, and $\|h_{m}\| < K$.  Thus $\|h\| \geq (m-1) \cdot K$, so $m \leq (\|h\|/K)+1$.  

The starting point of such a subpath is the center of a ball of outward radius $K$ which misses $\gamma$ by construction. Hence the diameter of the projection of $h_{i}$ to $\gamma$ is no more than $D$. Stringing together the projections of the $h_i$ we have a path from $x'$ to $y'$ of length at most $Dm \le D (\|h\|/K +1)$. As $\widebar \gamma $ lies between $x'$ and $y'$, and $\gamma$ is a directed geodesic, we have

\begin{align*}
1/3 T &\le D (\|h\|/K +1) \le D(\|p\|/K +1) \\
&\le D(C/K \cdot T +1).
\end{align*}

Rearranging gives,
\[
K \le \frac{3DCT}{T-3D}.
\]

Hence, so long as $T\ge 6D$, we have that $K \le 6CD$. Then by definition of $K$, the constant $L$ is bounded by $ 24C^2D(1+M_\epsilon)$. We conclude that there is an $x \in p$ such that $d(x ,\widebar \gamma) \le 24C^D(1+M_\epsilon)$. Since both $x$ and $\gamma$ are in $\X^{\ge \epsilon}$, we conclude that there is $T/3 \le t \le 2T/3$ with $\symd(x,\gamma(t)) \le 24C^2DM_\epsilon(1+M_\epsilon)$.
\end{proof}

\bibliographystyle{alpha}
\bibliography{mrstab.bbl}

\end{document}